\documentclass{amsart}
\usepackage{amssymb}
\usepackage{amsmath}
\usepackage{extarrows}
\usepackage{bbm}
\usepackage{mathrsfs}
\allowdisplaybreaks[4]

\makeatletter

\newcommand{\Rmnum}[1]{\expandafter\@slowromancap\romannumeral #1@}
\makeatother

\usepackage{xcolor}
\usepackage{amssymb}
\usepackage{amsmath,amsthm}
\usepackage{mathrsfs}

\usepackage{latexsym,euscript,dsfont}
\usepackage{bbm}
\usepackage{colortbl}
\usepackage{cite}
\usepackage{amsmath,amsthm,amsfonts,amssymb}
\usepackage{latexsym}
\usepackage{enumerate}
\usepackage{mathrsfs}
\def\bc{\begin{center}}
\def\ec{\end{center}}
\def\be{\begin{equation}}
\def\ee{\end{equation}}

\def\N{\mathbb N}

\newtheorem{lem}{Lemma}[section]

\newtheorem{dfn}[lem]{Definition}
\newtheorem{pro}[lem]{Proposition}
\newtheorem{thm}[lem]{Theorem}

\numberwithin{equation}{section}

\usepackage{colortbl}
\begin{document}
\title[Hausdorff dimension of some exceptional sets in L\"{u}roth expansions]{Hausdorff dimension of some exceptional sets in L\"{u}roth expansions}

\author{Ao Wang $^{1}$, Xinyun Zhang$^{2\ast}$}
\address{$^{1}$ School  of  Mathematics  and  Statistics,  Huazhong  University  of Science  and  Technology, 430074 Wuhan, China}
\email{wang\_ao@hust.edu.cn}
\address{$^{2}$ School  of  Mathematics  and  Statistics,  Huazhong  University  of Science  and  Technology, 430074 Wuhan, China}
\email{xinyunzhang@hust.edu.cn}

\thanks{This work was supported by NSFC 12331005.}
\maketitle
\begin{abstract}
In this paper, we study the metrical theory of the growth rate of digits in L\"{u}roth expansions. More precisely, for $ x\in \left( 0,1 \right]  $, let  $ \left[ d_1\left( x \right) ,d_2\left( x \right) ,\cdots \right] $ denote the L\"{u}roth expansion of $ x $, we completely determine the Hausdorff dimension of the following sets
\begin{align*}
 	E_{\mathrm{sup}}\left( \psi \right)  =\Big\{ x\in \left( 0,1 \right] :\limsup\limits_{n\rightarrow \infty}\frac{\log d_n\left( x \right)}{\psi \left( n \right)}=1 \Big\} , 	
\end{align*}
\begin{align*}
	E\left( \psi \right) =\Big\{ x\in \left( 0,1 \right] :\lim_{n\rightarrow \infty}\frac{\log d_n\left( x \right)}{\psi \left( n \right)}=1 \Big\} 	
\end{align*}
and
\begin{align*}
	E_{\mathrm{inf}}\left( \psi \right) =\Big\{ x\in \left( 0,1 \right] :
\liminf_{n\rightarrow \infty}\frac{\log d_n\left( x \right)}{\psi \left( n \right)}=1 \Big\} ,
\end{align*}
where $ \psi :\mathbb{N} \rightarrow \mathbb{R} ^+ $ is an arbitrary function satisfying $ \psi \left( n \right) \rightarrow \infty$ as $n\rightarrow \infty$.
\end{abstract}

\section{introduction}
 In 1883, L\"{u}roth\cite{luroth1883ueber} introduced the  L\"{u}roth  expansion. The L\"{u}roth transformation $ T:\left( 0,1 \right] \rightarrow \left( 0,1 \right]  $ is defined as follows:
\begin{align} \label{def:Luroth map}
	T\left( x \right) :=d_1\left( x \right) \left( d_1\left( x \right) -1 \right) \Big( x-\frac{1}{d_1\left( x \right)} \Big) ,
\end{align}
where $ d_1\left( x \right) =\lfloor \frac{1}{x} \rfloor +1$ ($  \lfloor \cdot \rfloor   $ denotes the largest integer not exceeding a real number). Through this algorithm very real number $  x\in \left( 0,1 \right] $ can be uniquely expanded into an infinite seris as
\begin{align} \label{def:luroth expansions}
	x=&\frac{1}{d_1(x)}+\frac{1}{d_1(x)\left( d_1(x)-1 \right) d_2(x)}+\cdots  \notag \\
	&+\frac{1}{d_1(x)\left( d_1(x)-1 \right) \cdots d_{n-1}(x)\left( d_{n-1}(x)-1 \right) d_n(x)}+\cdots,
\end{align}
where $ d_n\left( x \right) =d_1\left( T^{n-1}\left( x \right) \right) \left(  n\geqslant 2 \right)  $ . Such a series \eqref{def:luroth expansions} is called the L\"{u}roth expansion of $ x $ and written as
\begin{align*}
	x=\left[ d_1\left( x \right) ,d_2\left( x \right) ,\cdots \right] ,
\end{align*}
where $ d_n\left( x \right) \left( n\geqslant 1 \right)  $ are called the digits of the L\"{u}roth expansion of $ x $. L\"{u}roth\cite{luroth1883ueber} showed that each rational number has an infinite periodic expansion of the form \eqref{def:luroth expansions} and every irrational number has a unique infinite expansion.
For any $ n \in \mathbb{N}  $, the $ n $-th partial sum of \eqref{def:luroth expansions}
\begin{align*}
	\frac{p_n\left( x \right)}{q_n\left( x \right)} : = &\frac{1}{d_1(x)}+\frac{1}{d_1(x)\left( d_1(x)-1 \right) d_2(x)}+\cdots  \notag \\
	&+\frac{1}{d_1(x)\left( d_1(x)-1 \right)\cdots d_{n-1}(x)\left( d_{n-1}(x)-1 \right) d_n(x)}
\end{align*}
is called the $ n $-th convergent of the  L\"{u}roth expansion of $ x $.

The metrical theory of L\"{u}roth expansions is one of the major subjects in the study of L\"{u}roth expansions. It concerns the properties of the digits for almost all $x\in(0, 1]$. For some certain sets  defined in terms of the frequency of digits in L\"{u}roth expansions, Barreira, Iommi\cite{B.I09} and Fan et al.\cite{fan2010dimension} completely determined their Hausdorff dimension. In \cite{X.S18}, Sun and Xu studied the asymptotic behaviour of the maximal run-length function in L\"{u}roth expansions, and obtained its metrical property.
The following Borel-Bernstein theorem is important for the study of the metrical
theory of Lüroth expansions.
\begin{thm}[Borel-Bernstein]
	Let $ \phi \,\,: \mathbb{N} \rightarrow \mathbb{R} ^+  $ be an arbitrary function  and
	\begin{align*}
		F_{\phi}=\Big\{ x\in \left( 0,1 \right] \,\,: d_n\left( x \right) \geqslant \phi \left( n \right) \,\,{\rm{for\ infinitely\ many}}\,\,n\in \mathbb{N} \Big\} ,
	\end{align*}
	then, $  F_{\phi}  $  has Lerbesgue measure zero if $ \sum\limits_{n=1}^{\infty}{\frac{1}{\phi \left( n \right)}} $ converges and has full Lebesgue measure otherwise.
\end{thm}

The earliest research for exceptional sets of Borel-Bernstein Theorem in L\"{u}roth expansions was considered by $\rm{\check{S}}$al\'at \cite{salatZurMetrischenTheorie1968}. For any $ k\in \mathbb{N} $, he obtained the Hausdorff dimension of the set
\begin{align*}
M_k=\Big\{ x\in (0,1]:d_n(x)=k,n=1,2,... \Big\}.
\end{align*}

In \cite{shen2011fractional}, Shen and Fang investigated the Hausdorff dimension of the Cantor set
\begin{align*}
\Big\{ x\in \left( 0,1 \right] : d_n\left( x \right) \geqslant \varphi \left( n \right) \,\,{\rm{for\ any}} \  n\geqslant 1 \Big\}  ,
\end{align*}
where $ \varphi \left( n \right) : \mathbb{N} \rightarrow \mathbb{R} ^+ $ satisfying $ \lim\limits_{n\rightarrow \infty }\varphi \left( n \right) =\infty $.  Furthermore,  Shen\cite{shenHausdorffDimensionSet2017} studied the size of the set
\begin{align*}
\Big\{ x\in \left( 0,1 \right] \,:\, d_n\left( x \right) \geqslant \varphi \left( n \right) \,\,{\rm{for\ infinitely\ many}}\, n \in \mathbb{N}\Big\} 	
\end{align*} in the sense of the Hausdorff dimension.


The L\"{u}roth expansion has a tight connection with classical Diophantine approximation,
which focus on the question of how well a real number can be approximated
by rational numbers. For instance,  Cao, Wu, Zhang\cite{cao2013efficiency} proved that the Jarn\'{\i}k-like set

$$\mathcal {J}(\beta):=\Big\{x\in(0, 1]: \Big|x-\frac{p_n (x)}{q_n (x)}\Big|<\frac{1}{q_n (x)^{(\beta+1)}}\ {\rm for\ infintely\ many}\ n\in\N\Big\} $$
and the set
\begin{align*}
\widetilde{\mathcal{J}}(\beta):=\Big\{x\in(0, 1] : \limsup\limits_{n\rightarrow\infty}\frac{\log d_n (x) d_{n+1}(x)}{\log q_n (x)}\ge\beta\;\Big\}
\end{align*}
share the Hausdorff dimension when $\beta\ge0$. That is, $ {\rm dim_{H}}\,\mathcal{J}(\beta)={\rm dim_{H}}\,\widetilde{\mathcal{J}}(\beta)=\frac{1}{\beta+1}  $, where ${\rm dim_{H}}$ denotes the Hausdorff dimension. Furthermore, Tan and Zhou\cite{tan2021dimension} studied the metircal properties of the sets

$$\mathcal{J}(\tau):=\Big\{x\in (0, 1]:\Big|x-\frac{p_n (x)}{q_{n} (x)}\Big|<\frac{1}{q_n (x)^{(\beta+1)}}\ {\rm for}\ n\in\N\ {\rm ultimately}\Big\} $$
and
\begin{align*}
	\widetilde{\mathcal{J}}'(\tau ):=\Big\{x\in(0, 1]:\lim\limits_{n\rightarrow\infty}\frac{\log d_n (x) d_{n+1}(x)}{\log q_n (x)}=\tau \Big\}.
\end{align*}
More precisely, they proved that ${\rm dim_{H}}\,\mathcal{J}'(\tau)={\rm dim_{H}}\,\widetilde{\mathcal{J}}'(\tau)$ when $ \tau \ge 0$. And the dimension admits a dichotomy : it is either $ \frac{1}{\tau +2} $ or $1$  according to $ \tau >0 $ or $ \tau  =0 $.

In this paper, we continue the study of exceptional sets of Borel-Bernstein Theorem in L\"{u}roth expansions. Firstly, we investigate the Hausdorff dimension of the set
\begin{align*}
E_{\mathrm{sup}}\left( \psi \right)  =\Big\{ x\in \left( 0,1 \right] :\limsup_{n\rightarrow \infty}\frac{\log d_n\left( x \right)}{\psi \left( n \right)}=1 \Big\} 	.
\end{align*}
\begin{thm} \label{thm1.1}
	Let $ \psi :\mathbb{N} \rightarrow \mathbb{R} ^+ $ be an arbitrary function satisfying $ \lim\limits_{n\rightarrow \infty} \psi \left( n \right) =\infty   $ ,
	we have \par
	$ ( \mathrm{i} )  $ if $ \lim\limits_{n\rightarrow \infty} \frac{\psi \left( n \right)}{n}=0 $, then $ \mathrm{dim}_{\mathrm{H}}\,E_{\mathrm{sup}}\left( \psi \right)  =1 $. \par
	$ ( \mathrm{ii} )  $ if $ \underset{n\rightarrow \infty}{\lim  \mathrm{inf}}\frac{\psi \left( n \right)}{n} = \alpha \left( 0<\alpha <\infty \right)  $  ,
	then $ \mathrm{dim_H}\,E_{\mathrm{sup}}\left( \psi \right)  =H\left( \alpha \right)  $, where
	\begin{align*}
		H\left( \alpha \right) =\mathrm{dim_H}\big\{ x\in \left( 0,1 \right] \,\,: d_n\left( x \right) \geqslant \mathrm{e}^{\alpha n}\,\,{\rm for\ infinitely\ many}\,\,n\in \mathbb{N} \big\} .	
	\end{align*}
	\par
	$ ( \mathrm{iii} )  $ if $ \lim\limits_{n\rightarrow \infty} \frac{\psi \left( n \right)}{n}=\infty $,
	then $ \mathrm{dim_H}\,E_{\mathrm{sup}}\left( \psi \right)  =\frac{1}{1+{A}} $, where A is given by \begin{align*}
		\log A :=\liminf \limits_{n\rightarrow \infty}\frac{\log \psi \left( n \right)}{n}.
	\end{align*}	
\end{thm}
We also consider the Hausdorff demension of the sets
\begin{align*}
	E\left( \psi \right) =\Big\{ x\in \left( 0,1 \right] :\lim_{n\rightarrow \infty}\frac{\log d_n\left( x \right)}{\psi \left( n \right)}=1 \Big\} 	
\end{align*}
and
\begin{align*}
		E_{\mathrm{inf}}\left( \psi \right) =\Big\{ x\in \left( 0,1 \right] :
\liminf_{n\rightarrow \infty}\frac{\log d_n\left( x \right)}{\psi \left( n \right)}=1 \Big\} .
\end{align*}


\begin{thm} \label{Thm 1.2}
	Let $ \psi :\mathbb{N} \rightarrow \mathbb{R} ^+ $ be an arbitrary function satisfying $ \lim\limits_{n\rightarrow \infty} \psi \left( n \right) =\infty   ,$
    then
	\begin{align*}
		\mathrm{dim_H}E\left( \psi \right) =\frac{1}{2+\eta},
	\end{align*}
	where $ \eta $ is defined as
	\begin{align*}
		\eta :=\limsup \limits_{n\rightarrow \infty}\frac{\psi \left( n+1 \right)}{\psi \left( 1 \right) +\cdots +\psi \left( n \right)}.
	\end{align*}
\end{thm}
\begin{thm} \label{Thm 1.3}
	Let $ \psi :\mathbb{N} \rightarrow \mathbb{R} ^+ $ be an arbitrary function satisfying $ \lim\limits_{n\rightarrow \infty} \psi \left( n \right) =\infty   ,$
	then
	\begin{align*}
		\mathrm{dim_H}\,E_{\mathrm{inf}}\left( \psi \right) =\frac{1}{1+V},
	\end{align*}
	where $ V $ is given by
	\begin{align*}
		\log V:=
\limsup\limits_{n\rightarrow \infty}\frac{\log \psi \left( n \right)}{n}.
	\end{align*}
\end{thm}
The paper is organized as follows. In the next section, we present some elementary properties of L\"{u}roth expansions and some auxiliary lemmas. Section 3 is devoted to showing the proof of Theorem \ref{thm1.1}. In Section 4, we prove Theorem \ref{Thm 1.2} and Theorem \ref{Thm 1.3}.


\section{Preliminaries}
In this section, we racall some elementary properties of L\"{u}roth expansions and some useful lemmas that will be used later. \par
\begin{dfn}
\label{dfn:admissible}
A subsequence $ \left\{ d_n \right\} _{n\geqslant 1} $ of  $ \mathbb{N}  $ is said to be admissible  if there exists $ x\in \left( 0,1 \right]  $ such that
\begin{align*}
d_n\left( x \right) =d_n \, \,\mathrm{for} \,\,  n\geqslant 1.
\end{align*}
\end{dfn}
\begin{lem}[See \cite{galambosRepresentationsRealNumbers1976ar}]
\label{lem:admissible}
A sequence $ \left\{ d_n \right\} _{n\geqslant 1} $ of positive integers is
admissible if and only if
\begin{align*}
	d_n \geqslant 2   \,\, \mathrm{for} \,\,  n\geqslant 1.
\end{align*}
\end{lem}
For $n\geq1$, we denote by $ \Sigma _n $ the collection of all admissible words of length $ n $, i.e.,
\begin{align*}
	\Sigma _n=\big\{ \left( d_1,\cdots ,d_n \right) :d_j\geqslant 2 \,\, \mathrm{for} \,\, 1\leqslant j\leqslant n \big\} .
\end{align*}
\begin{dfn}
\label{dfn:cylinder set}
For any $ n\geqslant 1 $ and $ \left( d_1,\cdots ,d_n \right) \in \Sigma _n $, define
\begin{align*}
	I(d_1,\cdots ,d_n)=\big\{x\in \left( 0,1 \right]:d_1\left( x \right) =d_1,\ldots, d_n(x)=d_n \big\}
\end{align*}
and call it a $ n $-th order cylinder of L\"{u}roth expansion.
\end{dfn}
\begin{pro}[See \cite{galambosRepresentationsRealNumbers1976ar}]
\label{proposition: lenth of In}
	The $ n $-th order  cylinder $ I(d_1,\cdots ,d_n) $ is the interval with the left endpoint
	\begin{align*}
	\frac{1}{d_1}+\frac{1}{d_1(d_1-1)d_2}+\cdots +\prod_{k=1}^{n-1}{\frac{1}{d_k(d_k-1)}}\frac{1}{d_n}	
	\end{align*} and right endpoint
   \begin{align*}
   	\frac{1}{d_1}+\frac{1}{d_1(d_1-1)d_2}+\cdots +\prod_{k=1}^{n-1}{\frac{1}{d_k(d_k-1)}}\frac{1}{d_n}+\prod_{k=1}^n{\frac{1}{d_k(d_k-1)}}.
   \end{align*}
	Futhermore, one has
	\begin{align} \label{equation lenth of In}
		\big| I(d_1,\cdots ,d_n) \big|=\prod_{k=1}^n{\frac{1}{d_k(d_k-1)}}.
	\end{align}
In this context, we use the symbol $ ``| \cdot |" $ to denote the length of an interval.
\end{pro}
For convinience, we write
\begin{align*}
Q_n\left( x \right): =\prod_{k=1}^n{d_k\left( x \right) (d_k\left( x \right) -1)}.	
\end{align*}
\begin{lem}
\label{lem Qn}
	For any $ x\in \left( 0,1 \right]  $, $ Q_n\left( x \right) \geqslant 2^n . $
\end{lem}
\begin{proof}
Since $ d_n\left( x \right) \geqslant 2 $ for all $ n\geqslant 1 $, then
	\begin{align*}
	Q_n\left( x \right) =\prod_{k=1}^n{d_k\left( x \right) \big(d_k\left( x \right) -1\big)}\geqslant \prod_{k=1}^n{2\cdot (2-1)}=2^n .	
	\end{align*}
\end{proof}
For any integer $ M\geqslant 2 $, we set
\begin{align*}
	E_M=\Big\{ x\in \left( 0,1 \right] :2\leqslant d_n\left( x \right) \leqslant M\,\ \textup{for\ any}\ n\in \mathbb{N} \Big\} .
\end{align*}
The following lemma gives the Hausdorff dimension of $ E_M $.
\begin{lem}[See \cite{shen2008note}]
\label{lem: the hausdorff dimension of Em}
	For the L\"{u}roth expansion, let $ S\left( M \right)  $ be the solution of the equation
	\begin{align*}
		\sum_{k=2}^M{\left( \frac{1}{k\left( k-1 \right)} \right) ^s=1},
	\end{align*}
then $ \mathrm{dim_H}\,E_M=S\left( M \right)  $.
\end{lem}
\begin{lem}\label{lem:estimate of  HD }
For	$ S\left( M \right)  $  defined in Lemma \ref{lem: the hausdorff dimension of Em}, we have
	\begin{align*}
		0\leqslant S\left( M \right) <1 \,\, \text{\rm{and}} \,\,\lim_{M\rightarrow \infty} S\left( M \right) =1.
	\end{align*}
\end{lem}
\begin{proof}
	Consider
   \begin{align*}
		\sum_{k=2}^M{\frac{1}{k\left( k-1 \right)}}=1-\frac{1}{M}<1  \,\, \text{and}\,\, \sum_{k=2}^M{\left( \frac{1}{k\left( k-1 \right)} \right) ^0}=M-1\geqslant 1,
	\end{align*}
	then
	\begin{align*}
		0\leqslant S\left( M \right) <1
	\end{align*}
	since $ \big( \frac{1}{k\left( k-1 \right)} \big) ^s $ is decreasing with respect to $ s\in \left[ 0,\infty \right)$.	
		
For any $ \epsilon >0 $, we set $ m =\lceil 1+\frac{1}{2^{\epsilon}-1} \rceil  ,$ where $ \lceil \cdot \rceil  $  represents the samllest integer not less than a real number. Then for any $ M>m $, we have
	\begin{align*}
		\sum_{k=2}^M{\left( \frac{1}{k\left( k-1 \right)} \right) ^{1-\epsilon}}
		&=\sum_{k=2}^M{\left( \frac{1}{k\left( k-1 \right)} \right)}\big( k\left( k-1 \right) \big) ^{\epsilon}	\\
		& >2^{\epsilon}\sum_{k=2}^M{\frac{1}{k\left( k-1 \right)}} \\
		& =2^{\epsilon}\left( 1-\frac{1}{M} \right) \\
		&=1.
	\end{align*}
Hence $ S\left( M \right) >1-\epsilon  $. Conbine this and $ S\left( M \right) <1 $, we have $ \lim\limits_{M\rightarrow \infty} S\left( M \right) =1 $.
\end{proof}

\begin{lem}
\label{lem:H(alpha)}
	For $ H\left( \alpha \right)  $ defined in Theorem \ref{thm1.1}, let $G\left( B \right)  $ be the unique solution of the equation
    \begin{align*}
	\sum_{k=2}^{\infty}{\left( \frac{1}{Bk\left( k-1 \right)} \right) ^s=1},
    \end{align*}
    then$$H\left( \alpha \right) =G\left( {e^{\alpha}} \right).$$
Furthermore,
$ H\left( \alpha \right)  $ is continuous with respect to $ \alpha \in \left( 0,\infty \right)  $ and we have
    \begin{align*}
    	\lim_{\alpha \rightarrow 0^+} H\left( \alpha \right) =1 \,,\,\lim_{\alpha \rightarrow \infty} H\left( \alpha \right) =\frac{1}{2}.
    \end{align*}
\end{lem}
\begin{proof}
	For any $ B>1 $, according to Lemma 2.3 of \cite{shenHausdorffDimensionSet2017}, $ G\left( B \right) $ is  continuous with respect to $ B \in \left( 1,\infty \right)  $  and \begin{align*}
		\lim_{B\rightarrow 1^+} G\left( B \right)=1\,\,\text{and}\,\,\lim_{B\rightarrow \infty} G\left( B \right)=\frac{1}{2}.
	\end{align*}
	Put $ B= e^{\alpha} >1 $, by Theorem 4.2 of \cite{shenHausdorffDimensionSet2017}, it follows that  $ H\left( \alpha \right) =G\left( {e^{\alpha}} \right) $, then $ H\left( \alpha \right)  $ is continuous with respect to $ \alpha \in \left( 0,\infty \right)  $. Moreover, we have
	\begin{align*}
		\lim_{\alpha \rightarrow 0^+} H\left( \alpha \right) =1  \,\,\text{\rm{and}}\,\,\lim_{\alpha \rightarrow \infty} H\left( \alpha \right) =\frac{1}{2}.
	\end{align*}

\end{proof}
The following results is useful to estimate the lower bound of Hausdorff dimension for some sets.
\begin{lem}[See \cite{arroyo2020hausdorff}]
	\label{lem:key lem}
	Let $ \left\{ u_n \right\} _{n\geqslant 1} $ be a subsequence of $ \mathbb{N} $ with $u_n\ge 4$ for any $ n\ge1 $ and $ \lim\limits_{n\rightarrow \infty} u_n=\infty  $, then for any integer $ N\ge2 $, one has
	\begin{align*}
		\mathrm{dim_H}\,\Big\{ x\in \left( 0,1 \right] :u_n\leqslant d_n\left( x \right) \leqslant Nu_n-1\ \textup{for\ any}\ n\in \mathbb{N} \Big\} =u_0,
	\end{align*}
	where
	\begin{align*}
		u_0 :=\liminf\limits_{n\rightarrow \infty}\frac{\log \left( u_1\cdots u_n \right)}{2\log \left( u_1\cdots u_n \right) +\log u_{n+1}}.
	\end{align*}
\end{lem}
\begin{pro}[See \cite{falconer2014fractal}]
	\label{lem:CSofHD}
	Let $ E\subseteq \mathbb{R} ^n $, suppose that $ f\,\,: E\rightarrow \mathbb{R} ^m $ satisfies the $ H\ddot{o}lder $ condition
	\begin{align*}
		\big| f \left( x\right)-f\left( y \right)\big |\le c\big|x-y\big|^{\alpha}\,\,    (x,y\in E),
	\end{align*}
where $ c $ is a constant. Then
	\begin{align*}
		\mathrm{dim_H}\,f\left( E \right) \le \frac{1}{\alpha}\mathrm{dim_H}\,E.
	\end{align*}
\end{pro}

\section{PROOF OF THEOREM \ref{thm1.1}}
Recall that
\begin{align*}
	E_{\mathrm{sup}}\left( \psi \right)  =\left\{ x\in \left( 0,1 \right]  :\limsup_{n\rightarrow \infty}\frac{\log d_n\left( x \right)}{\psi \left( n \right)}=1 \right\} .
\end{align*}

We consider the following cases:\\
{\bf Case \Rmnum{1}: } $ \underset{n\rightarrow \infty}{\lim}\,\frac{\psi \left( n \right)}{n}=0 $. \par
In this case, the upper bound follows trivially by noting that $ E_{\mathrm{sup}}\left( \psi \right) \subseteq \left( 0,1 \right]  $. We only focus on the lower bound of the Hausdorff dimension of $ E_{\mathrm{sup}}\left( \psi \right) $. The lower bound is obtained by constructing a Cantor subset $  E_M\left( \psi \right) $ of $ E_{\mathrm{sup}}\left( \psi \right) $.

Let $ \left\{ m_k \right\} _{k\geqslant 1} $ be a subequence of $\N$ with $ m_k=2^k $ for all $  k\geqslant 1 $, and for any integer $ M\ge2  $, define
\begin{align*}
	E_M\left( \psi \right)=\Big\{x\in \left( 0,1 \right] :d_{m_k}\left( x \right) =\lfloor \mathrm{e}^{\,\psi \left( m_k \right)}\, \rfloor +1 \,\, \text{and}\ 2\le d_i\left( x \right) \le M\,(i\ne m_k\,\,{\rm for\ any}\,\,k\ge 1)  \Big\}.
\end{align*}
It is clear that $  E_M\left( \psi \right) \subseteq E_{\mathrm{sup}}\left( \psi \right) . $

Now, we estimate the lower bound of the Huasdorff dimension of $ E_M\left( \psi \right) $. Firstly, we use the language of symbolic space to describe $ E_M\left( \psi \right) $. For any integer $ n\geqslant 1 $, define
\begin{align*}
C_n=\Big\{(\sigma _1,\cdots ,\sigma _n)\in \Sigma _n:\sigma _{m_k}=\lfloor \mathrm{e}^{\psi \left( m_k \right)} \rfloor +1 \, \, \text{and} \, \,2\le \sigma _i\le M,1\le i\ne m_k\le n\Big\}.
\end{align*}
For any integer $ n\geqslant 1$ and  $ (\sigma _1,\cdots ,\sigma _n)\in C_n $, the $ n  $-th basic interval  is defined as follows:
\begin{align} \label{found interval}
J_n(\sigma _1,\cdots ,\sigma _n)=\bigcup_{(\sigma _1,\cdots ,\sigma _n,\sigma _{n+1})\in C_{n+1}}{I_{n+1}(\sigma _1,\cdots ,\sigma _n,\sigma _{n+1})}
	.
\end{align}
Obviously,
\begin{align}\label{EmInJn}
	E_M\left( \psi \right) =\bigcap_{n\ge 1}{\bigcup_{(\sigma _1,\cdots ,\sigma _n)\in C_n}{I_n(\sigma _1,\cdots ,\sigma _n)}}=\bigcap_{n\ge 1}{\bigcup_{(\sigma _1,\cdots ,\sigma _n)\in C_n}{J_n(\sigma _1,\cdots ,\sigma _n)}}.
\end{align}

Let $ t\left( n \right) =\#\left\{ k\in \mathbb{N} \,\,: 2^k\leqslant n \right\}   $, where the symbol $``\#"$  denotes the cardinality of a set. It is easy to see that $t\left( n \right) \leqslant \lfloor \log _2n \rfloor \leqslant \log _2n $,  then we have
\begin{align}\label{eq:lim t/n}
	\underset{n\rightarrow \infty}{\lim}\frac{t\left( n \right)}{n}=0.	
\end{align}
For any $ (\sigma _1,\cdots ,\sigma _n)\in C_n $, let $ \left( \sigma _1,\cdots ,\sigma _n \right) ^*$ be the block obtained  by eliminating the terms $ \left\{ \sigma _{m_k}: 1\leqslant k\leqslant t\left( n \right) \right\}  $ in $ (\sigma _1,\cdots ,\sigma _n)$. Obviously,
\begin{align*}
\left( \sigma _1,\cdots ,\sigma _n \right) ^*\in \left\{ 2,3,\cdots ,M \right\} ^{n-t\left( n \right)}.
\end{align*}
For simplicity, we write
\begin{align}
	I_{n}^{*}(\sigma _1,\cdots ,\sigma _n):=\,\,I_{n-t\left( n \right)}\left( \sigma _1,\cdots ,\sigma _n \right) ^*.
\end{align}

Since $\lim\limits_{n\rightarrow \infty}\frac{\psi \left( n \right)}{n}=0 $, then
\begin{align*}
	0\leqslant \limsup_{{n\rightarrow \infty}}\frac{1}{n}\sum_{k=1}^{t\left( n \right)}{\psi \left( m_k \right)}\leqslant \limsup_{n\rightarrow \infty}\frac{1}{2^{t\left( n \right)}}\sum_{k=1}^{t\left( n \right)}{\psi \big( 2^k \big)}\leqslant \limsup_{n\rightarrow \infty}\,\,\frac{1}{2^n}\sum_{k=1}^n{\psi \left( 2^k \right)}=0,
\end{align*}
we have
\begin{align}\label{eqa:lim:t(n)andphi}
	\underset{n\rightarrow \infty}{\lim}\frac{1}{n}\sum_{k=1}^{t\left( n \right)}{\psi \left( m_k \right)}=0.
\end{align}

Now we compare the length of $ I_n\left( \sigma _1,\cdots ,\sigma _n \right)  $ and the length of $ I_{n}^{*}(\sigma _1,\cdots ,\sigma _n) $.
\begin{lem} \label{lem In<In-}
	For any $ 0<\epsilon <1 $, there exists a positive integer $ N_1 $ such that for any $ n\geqslant N_1 $ and $ (\sigma _1,\cdots ,\sigma _n)\in C_n $, we have
	\begin{align*}
     \big| I_n \left( \sigma _1,\cdots ,\sigma _n\right)\big|\ge \big|I_{n}^{*}(\sigma _1,\cdots ,\sigma _n)\big|^{1+\epsilon}.
	\end{align*}
\end{lem}
\begin{proof}
	For any $ (\sigma _1,\cdots ,\sigma _n)\in C_n $, set $ Q_0:=1 $, $ Q_n:=\prod\limits_{k=1}^n{\sigma _k(\sigma _k-1)} $ and $ Q_n\left( \sigma _1,\cdots ,\sigma _n \right) ^*:=Q_n/\prod\limits_{i=1}^{t\left( n \right)}{\sigma _{m_i}(\sigma _{m_i}-1)}. $
According to Proposition \ref{proposition: lenth of In}, we have
	\begin{align*}
		\big| I_n(\sigma _1,\cdots ,\sigma _n) \big|&=\frac{1}{Q_n}
	\end{align*}
and	
    \begin{align*}
		\big| I_{n}^{*}(\sigma _1,\cdots ,\sigma _n) \big|&=\frac{1}{Q_n\left( \sigma _1,\cdots ,\sigma _n \right) ^*}.
	\end{align*}
By the definition of $ C_n $, one has
	\begin{align*}
		e^{\psi \left( m_k \right)}<\sigma _{m_k}\leqslant e^{\psi \left( m_k \right)}+1 \,\,\text{for\,\,any\,\,} k\geqslant1,
	\end{align*}
then
	\begin{align} \label{eq :lem 3.2one}
			\big| I_n(\sigma _1,\cdots ,\sigma _n) \big|
			=\frac{1}{Q_ n}
			&=\frac{1}{Q_n\left( \sigma _1,\cdots ,\sigma _n \right) ^*\cdot \prod\limits_{i=1}^{t\left( n \right)}{\sigma _{m_i}(\sigma _{m_i}-1)}} \notag \\
			&\geqslant \frac{1}{Q_n\left( \sigma _1,\cdots ,\sigma _n \right) ^*}\cdot \prod\limits_{i=1}^{t\left( n \right)}{\frac{1}{\left( e^{\psi \left( m_i \right)}+1 \right) \cdot \mathrm{e}^{\psi \left( m_i \right)}}}	.	
	\end{align}
Set
$$b_n =\frac{t\left( n \right) -1}{n}\cdot \log 2+\frac{2}{n}\cdot \sum_{k=1}^{t\left( n \right)}{\psi \left( m_k \right)}$$
and
$$
		c_n=\frac{\log 2}{n}+\frac{t\left( n \right)}{n}\cdot \log 2 ,
$$
then by \eqref{eq:lim t/n} and \eqref{eqa:lim:t(n)andphi}, we have$$
		\underset{n\rightarrow \infty}{\lim}b_n =\underset{n\rightarrow \infty}{\lim}\Big( \frac{t\left( n \right) -1}{n}\cdot \log 2+\frac{2}{n}\cdot \sum_{k=1}^{t\left( n \right)}{\psi \left( m_k \right)} \Big) =0$$
and 	
		$$\underset{n\rightarrow \infty}{\lim}c_n=\underset{n\rightarrow \infty}{\lim}\left( \frac{\log 2}{n}+\frac{t\left( n \right)}{n}\cdot \log 2 \right) =0 .$$
Hence for any $ 0<\epsilon <1 $, there exists an integer $ N_1 $ such that for any $ n\geqslant N_1   $, one has
	\begin{align}\label{eq:bn}
		\frac{t\left( n \right) -1}{n}\cdot \log 2+\frac{2}{n}\cdot \sum_{k=1}^{t\left( n \right)}{\psi \left( m_k \right)}<\epsilon \cdot \log a
	\end{align}
and
	\begin{align} \label{eq:cn}
		\frac{\log 2}{n}+\frac{t\left( n \right)}{n}\epsilon \cdot \log 2<\frac{\log 2}{n}+\frac{t\left( n \right)}{n}\cdot \log 2<\epsilon \left( \log 2-\log a \right) ,
	\end{align}
where $ a $ is a real number with $ 0<a<1$. This shows that
	\begin{align}\label{eq bcn fangsuo}
		2^{t\left( n \right)}\cdot \mathrm{e}^{2 \cdot\sum\limits_{k=1}^{t\left( n \right)}{\psi \left( m_k \right)}} \,\overset{\textup{by}\,\eqref{eq:bn}}{<}2a^{n\epsilon}\,\overset{\textup{by} \,\eqref{eq:cn}}{<}2^{\left( n-t\left( n \right) \right)  \epsilon}.
	\end{align}
Note that
   \begin{align}\label{eq:en}
\prod_{i=1}^{t\left( n \right)}{\left( \mathrm{e}^{\psi \left( m_k \right)}+1 \right) \cdot \mathrm{e}^{\psi \left( m_k \right)}}\leqslant \prod_{i=1}^{t\left( n \right)}{2\cdot \mathrm{e}^{2\psi \left( m_k \right)}}=2^{t\left( n \right)}\cdot \mathrm{e}^{2\cdot\sum\limits_{k=1}^{t\left( n \right)}{\psi \left( m_k \right)} } ,
	\end{align}
	then by formula \eqref{eq :lem 3.2one}, \eqref{eq bcn fangsuo}, \eqref{eq:en} and Lemma \ref{lem Qn}, when $n\geqslant N_1  $,
	\begin{align*}
		\big| I_n(\sigma _1,\cdots ,\sigma _n) \big|
		&\geqslant \frac{1}{Q_n\left( \sigma _1,\cdots ,\sigma _n \right) ^*\cdot \prod\limits_{i=1}^{t\left( n \right)}{\left( \mathrm{e}^{\psi \left( m_k \right)} +1 \right) \cdot \mathrm{e}^{\psi \left( m_k \right)}}}
		\\
		&\geqslant \frac{1}{Q_n\left( \sigma _1,\cdots ,\sigma _n \right) ^*\cdot 2^{\left( n-t\left( n \right) \right) \epsilon}}
		\\
		&\geqslant \frac{1}{Q_n\left( \sigma _1,\cdots ,\sigma _n \right) ^*\cdot \left( Q_n\left( \sigma _1,\cdots ,\sigma _n \right) ^* \right) ^{\epsilon}} \\
		& =\big| I_{n}^{*}(\sigma _1,\cdots ,\sigma _n) \big|^{1+\epsilon}.
	\end{align*}
\end{proof}


In order to make use of Proposition \ref{lem:CSofHD} to estimate the Hausdorff dimension of $E_M\left( \psi \right)$, now  we estimate $|x-y|,$ where $ x\ne y\in E_M\left( \psi \right) $.
\begin{lem}\label{lem: 3.3}
	For any $ x\ne y\in E_M\left( \psi \right) $, there exists $ (\sigma _1,\cdots ,\sigma _n)\in C_n $ and $ l_{n+1}\ne r_{n+1} $ satisfying $ (\sigma _1,\cdots ,\sigma _n,l_{n+1})\in C_{n+1} $ and  $ (\sigma _1,\cdots ,\sigma _n,r_{n+1})\in C_{n+1} $  such that
	\begin{align*}
	x\in I_{n+1}(\sigma _1,\cdots ,\sigma _n,l_{n+1}) ,\  y\in I_{n+1}(\sigma _1,\cdots ,\sigma _n,r_{n+1})
    \end{align*}
and
	\begin{align*}
	\big| x-y \big|\geqslant \frac{1}{M^3}\big| I_n(\sigma _1,\cdots ,\sigma _n) \big|.
    \end{align*}	
\end{lem}
\begin{proof}
	For any $ x=\left[ x_1,x_2,\cdots \right] \in E_M\left( \psi \right)  $, $ y=\left[ y_1,y_2,\cdots \right] \in E_M\left( \psi \right)  $ and $ x\ne y $, suppsoe that $ n $ is the greatest integer such that $ x_i=y_i $ for $ 1\leqslant i\leqslant n $, then there exists $ (\sigma _1,\cdots ,\sigma _n)\in C_n $ such that $ x,y\in I_n(\sigma _1,\cdots ,\sigma _n) $. If $ n=0 $, set $ I_0:=\left( 0,1 \right]  $, then $ x,y\in I_0 $. Otherwise, there exsit $ l_{n+1}\ne r_{n+1} $ such that $ x\in I_{n+1}(\sigma _1,\cdots ,\sigma _n,l_{n+1})  $ and $ y\in I_{n+1}(\sigma _1,\cdots ,\sigma _n,r_{n+1}) $, where $ (\sigma _1,\cdots ,\sigma _n,l_{n+1})\in C_{n+1} $ and $ (\sigma _1,\cdots ,\sigma _n,r_{n+1})\in C_{n+1} $. \par
	Without loss of generality, we assume $ x<y$. then $ l_{n+1}>r_{n+1} $, Note that for all $ k\in \mathbb{N} $, $ n+1\ne m_k ,$ otherwise $ l_{n+1}=r_{n+1}=\lfloor \mathrm{e}^{\psi \left( m_k \right)} \rfloor +1  $. Recall that $ x\in J_{n+1}(\sigma _1,\cdots ,\sigma _n,l_{n+1})\cap E_M\left( \psi \right)  $ and $ y\in J_{n+1}(\sigma _1,\cdots ,\sigma _n,r_{n+1})\cap E_M\left( \psi \right)  $, then
	$ | x-y | $ is not less than the gaps between $ J_{n+1}(\sigma _1,\cdots ,\sigma _n,l_{n+1}) $ and $ J_{n+1}(\sigma _1,\cdots ,\sigma _n,r_{n+1}) $. \par
	We consider the following two cases:\\
	$ \left( \mathrm{i} \right)  $ If $ n+2=m_k $ for some $ k\in \mathbb{N}  $.\par
	In this case, $y-x $ is greater than the distance between the left endpoint of the basic interval $ J_{n+1}(\sigma _1,\cdots ,\sigma _n,r_{n+1}) $ and the right endpoint of the basic interval $ J_{n+1}(\sigma _1,\cdots ,\sigma _n,l_{n+1}) $. Let $ t=\lfloor \mathrm{e}^{\psi \left( m_k \right)} \rfloor +1\geqslant 2 $, then
	\begin{align*}
		J_{n+1}(\sigma _1,\cdots ,\sigma _n,l_{n+1})&=I_{n+2}(\sigma _1,\cdots ,\sigma _n,l_{n+1},t) ,\\
		J_{n+1}(\sigma _1,\cdots ,\sigma _n,r_{n+1})&=I_{n+2}(\sigma _1,\cdots ,\sigma _n,r_{n+1},t).
	\end{align*}
	For simplicity, we write $ s_0:=1 $, $ s_k=\sigma _k\left( \sigma _k-1 \right)   $ for $ 1\leqslant k\leqslant n $, recall that $ Q_k=s_0s_1\cdots s_{k-1}s_k $ for $ 1\leqslant k\leqslant n $, then by Proposition \ref{proposition: lenth of In}, we have	
	\begin{align*}
	\big| y-x \big|
	&\geqslant \left( \sum_{k=1}^n{\frac{1}{Q_{k-1}}}\cdot \frac{1}{\sigma _k}+\frac{1}{Q_n}\cdot \frac{1}{r_{n+1}}+\frac{1}{Q_n}\cdot \frac{1}{r_{n+1}\left( r_{n+1}-1 \right)}\cdot \frac{1}{t} \right)    \\
	&\quad -\left( \sum_{k=1}^n{\frac{1}{Q_{k-1}}}\cdot \frac{1}{\sigma _k}+\frac{1}{Q_n}\cdot \frac{1}{l_{n+1}}+\frac{1}{Q_n}\cdot  \frac{1}{l_{n+1}\left( l_{n+1}-1 \right)}\cdot \frac{1}{t} + \right. \\
	& \quad \qquad \left. \frac{1}{Q_n}\cdot \frac{1}{l_{n+1}\left( l_{n+1}-1 \right)}\cdot \frac{1}{t\left( t-1 \right)} \right)  \\
	&=\frac{1}{Q_n}\left( \frac{1}{r_{n+1}}-\frac{1}{l_{n+1}}+\frac{1}{r_{n+1}\left( r_{n+1}-1 \right)}\cdot \frac{1}{t}-\frac{1}{l_{n+1}\left( l_{n+1}-1 \right)}\cdot \frac{1}{t-1} \right)  \\
	& \geqslant \frac{1}{Q_n}\left( \frac{l_{n+1}-r_{n+1}}{r_{n+1}l_{n+1}}+\frac{1}{l_{n+1}\left( l_{n+1}-1 \right)}\left( \frac{1}{t}-\frac{1}{t-1} \right) \right)  \\
	& \geqslant \frac{1}{Q_n}\left( \frac{1}{r_{n+1}l_{n+1}}-\frac{1}{l_{n+1}\left( l_{n+1}-1 \right)}\cdot \frac{1}{2} \right)  \\
	&\geqslant \frac{1}{Q_n}\frac{1}{l_{n+1}\left( l_{n+1}-1 \right)}\cdot \frac{1}{2} \\
	&\geqslant \frac{1}{Q_n}\frac{1}{2M\left( M-1 \right)} \\
	& \geqslant \frac{1}{M^3}\big| I_n(\sigma _1,\cdots ,\sigma _n) \big|.
\end{align*}
	$ \left( \mathrm{ii} \right)  $ If $ n+2 \ne m_k $ for any $ k\in \mathbb{N} $.\par
	In this case, $y-x $ is greater than the distance between the left endpoint of $ J_{n+1}(\sigma _1,\cdots ,\sigma _n,r_{n+1}) $ and the right point of $ J_{n+1}(\sigma _1,\cdots ,\sigma _n,l_{n+1}) $. Note that $ n+2 \ne m_k $ for any $ k\ge1 $, then  the left endpoint of $ J_{n+1}(\sigma _1,\cdots ,\sigma _n,r_{n+1})  $ is just the left endpoint of $ I_{n+2}(\sigma _1,\cdots ,\sigma _n,r_{n+1},M)  $ and the right point of $ J_{n+1}(\sigma _1,\cdots ,\sigma _n,l_{n+1}) $ is the right point of $ I_{n+2}(\sigma _1,\cdots ,\sigma _n,l_{n+1},2) $. Similar to case  $ \left( \mathrm{i} \right)  $, we have
	\begin{align*}
	\big| y-x \big|
	& \geqslant
	\frac{1}{Q_n}\left( \frac{1}{r_{n+1}}-\frac{1}{l_{n+1}}+\frac{1}{r_{n+1}\left( r_{n+1}-1 \right)}\cdot \frac{1}{M}-\frac{1}{l_{n+1}\left( l_{n+1}-1 \right)} \right) \\
	& \geqslant \frac{1}{Q_n}\frac{1}{l_{n+1}\left( l_{n+1}-1 \right)}\cdot \frac{1}{M}  \\
	& \geqslant \frac{1}{Q_n}\frac{1}{M^2\left( M-1 \right)} \\
	& \geqslant \frac{1}{M^3}\big| I_n(\sigma _1,\cdots ,\sigma _n) \big|.
\end{align*}
\end{proof}
Recall that
\begin{align*}
	E_M=\Big\{ x\in \left( 0,1 \right] :2\leqslant d_n\left( x \right) \leqslant M\,\, {\,\,\rm{for} \,\,\rm{any}}\,\ n\in \mathbb{N} \Big\} 	
\end{align*}
and write
\begin{align*}
\left[ d_1,d_2,\cdots d_n \right] =\frac{1}{d_1}+\frac{1}{d_1\left( d_1-1 \right) d_2}+...+\frac{1}{d_1\left( d_1-1 \right) \cdots d_{n-1}\left( d_{n-1}-1 \right) d_n}.	
\end{align*}
Define a map $ f:E_M\left( \psi \right) \rightarrow E_M $ as follows: for any $ x=\left[ x_1,x_2,\cdots \right] \in E_M\left( \psi \right)  $, let
\begin{align*}
f\left( x \right) =\underset{n\rightarrow \infty}{\lim}\left[ x_1,\cdots ,x_n \right] ^*,
\end{align*}
where $ \left[ x_1,\cdots ,x_n \right] ^* $ is the real number obtained  by eliminating the terms $ \left\{ x _{m_k}: 1\leqslant k\leqslant t\left( n \right) \right\}  $ in $ \left[ x_1,\cdots ,x_n \right]$. For any $ x=\left[ x_1,x_2,\cdots \right] \in E_M\left( \psi \right)   $, $ y=\left[ y_1,y_2,\cdots \right] \in E_M\left( \psi \right)  $, let $ m $ be the greatest integer such that $ x_i=y_i$ for $ 1\leqslant i\leqslant m  $, it is easy to see that for any $  0<\epsilon <1 $, if
\begin{align}\label{eq:hype}
	\big| x-y \big|<\frac{1}{M^3}\underset{(\sigma _1,\cdots ,\sigma _{N_1})\in C_{N_1}}{\min}\left\{ \big| I_{N_1}(\sigma _1,\cdots ,\sigma _{N_1}) \big| \right\} ,	
\end{align}
then $ m>N_1  $, where $ N_1 $ is the same as $ N_1 $ in Lemma \ref{lem In<In-}. Otherwise, by Lemma \ref{lem: 3.3}, one has
\begin{align*}
	\big| x-y \big|\geqslant \frac{1}{M^3}\big| I_m(\sigma _1,\cdots ,\sigma _m) \big|\geqslant \frac{1}{M^3}\underset{(\sigma _1,\cdots ,\sigma _{N_1})\in C_{N_1}}{\min}\left\{ \big| I_{N_1}(\sigma _1,\cdots ,\sigma _{N_1}) \big| \right\} ,
\end{align*}
which contradicts the hypothesis \eqref{eq:hype}. Furthermore, according to Lemma \ref{lem In<In-} and Lemma \ref{lem: 3.3}, we have
\begin{align}\label{eq:Holder}
	\big| f\left( x \right) -f\left( y \right) \big|\leqslant \big| I_{m}^{*}(\sigma _1,\cdots ,\sigma _m)\big|\leqslant \big| I_m(\sigma _1,\cdots ,\sigma _m) \big|^{\frac{1}{1+\epsilon}}\leqslant M^{\frac{3}{1+\epsilon}}\cdot \big| x-y \big|^{\frac{1}{1+\epsilon}}.
\end{align}
\par Take a sufficiently large $ k $ such that for any $ \left( \sigma _1,\cdots ,\sigma _k \right) \in C_k $, one has
\begin{align*}
	\big| I_k(\sigma _1,\cdots ,\sigma _k) \big|<\frac{1}{M^3}\underset{(\sigma _1,\cdots ,\sigma _{N_1})\in C_{N_1}}{\min}\left\{ \big| I_{N_1}(\sigma _1,\cdots ,\sigma _{N_1}) \big| \right\} .
\end{align*}
It is clear that
\begin{align*}
	E_M\left( \psi \right) =\bigsqcup_{(\sigma _1,\cdots ,\sigma _k)\in C_k}{E_M\left( \psi \right) \cap I_k(\sigma _1,\cdots ,\sigma _k)}.
\end{align*}
For simplicity, we write $ F\left( \sigma _1,\cdots ,\sigma _k \right) =E_M\left( \psi \right) \cap I_k(\sigma _1,\cdots ,\sigma _k) $, then by the countable stability of Hausdorff dimension, we have
\begin{align}\label{eq:sup of dim}
	\mathrm{dim}E_M\left( \psi \right) =\underset{\left( \sigma _1,\cdots ,\sigma _{k} \right) \in C_k}{\mathrm{sup}}\mathrm{dim_H}\,F\left( \sigma _1,\cdots ,\sigma _k \right) .
\end{align}
For any $ \left( \sigma _1,\cdots ,\sigma _k \right) \in C_k $, by Proposition \ref{lem:CSofHD} and \eqref{eq:Holder}, we have
\begin{align*}
	\mathrm{dim_H}\,F\left( \sigma _1,\cdots ,\sigma _k \right) \geqslant \frac{1}{1+\epsilon}\,\mathrm{dim_H}\,f\left( F\left( \sigma _1,\cdots ,\sigma _k \right) \right)  .
\end{align*}
Then by \eqref{eq:sup of dim}, one has
\begin{align}\label{eq:dim of Em(phi) fangsuo}
	\mathrm{dim_H}\,E_M\left( \psi \right)
	&\geqslant \underset{\left( \sigma _1,\cdots ,\sigma _k \right) \in C_k}{\mathrm{sup}}\frac{1}{1+\epsilon}\,\mathrm{dim_H}\,f\left( F\left( \sigma _1,\cdots ,\sigma _k \right) \right).
\end{align}
Note that
\begin{align*}
	E_M=f\left( E_M\left( \psi \right) \right) =f\Big( \bigsqcup_{\left( \sigma _1,\cdots ,\sigma _k \right) \in C_k}{F\left( \sigma _1,\cdots ,\sigma _k \right)} \Big) \subseteq  \bigsqcup_{\left( \sigma _1,\cdots ,\sigma _k \right) \in C_k}{f\left( F\left( \sigma _1,\cdots ,\sigma _k \right) \right)},
\end{align*}
then
\begin{align}\label{eq: dim of em fangsuo}
	\mathrm{dim_H}\,E_M\leqslant \mathrm{dim_H}\,\bigsqcup_{\left( \sigma _1,\cdots ,\sigma _k \right) \in C_k}{f\left( F\left( \sigma _1,\cdots ,\sigma _k \right) \right)}=\underset{\left( \sigma _1,\cdots ,\sigma _k \right) \in C_k}{\mathrm{sup}}\mathrm{dim_H}\,f\left( F\left( \sigma _1,\cdots ,\sigma _k \right) \right) .
\end{align}
Recall that $ E_M\left( \psi \right) \subseteq E_{\mathrm{sup}}\left( \psi \right)  $, by combining \eqref{eq:dim of Em(phi) fangsuo}, \eqref{eq: dim of em fangsuo}, we have
\begin{align*}
	\mathrm{dim_H}\,E_{\mathrm{sup}}\left( \psi \right)  \geqslant \mathrm{dim_H}\,E_M\left( \psi \right) \geqslant \frac{1}{1+\epsilon}\,\mathrm{dim_H}\,E_M.
\end{align*}

Let $ \epsilon \rightarrow 0^+ $, then
\begin{align*}
	\mathrm{dim_H}\,E_{\mathrm{sup}}\left( \psi \right)  \geqslant \mathrm{dim_H}\,E_M.
\end{align*}
Moreover, by Lemma \ref{lem: the hausdorff dimension of Em} and Lemma \ref{lem:estimate of  HD } and let $ M\rightarrow \infty  $, we have
\begin{align*}
	\mathrm{dim_H}\,E_{\mathrm{sup}}\left( \psi \right)  \geqslant \underset{M\geqslant 2}{\mathrm{sup}}\mathrm{dim_H}\,E_M=1 .
\end{align*}
{\bf Case \Rmnum{2}: }$ \liminf\limits_{n\rightarrow \infty}\frac{\psi \left( n \right)}{n}=\alpha \,\,\left( 0<\alpha <\infty \right)  .$ \par
We split the proof into two parts.\par
\textbf{Upper bound}  For any $x\in E_{\mathrm{sup}}\left( \psi \right)    $, it holds that $ \limsup\limits_{n\rightarrow \infty}\frac{\log d_n\left( x \right)}{\psi \left( n \right)}=1   $, then for any $ 0<\epsilon <\min \left\{ \alpha ,1 \right\}  $, we have
\begin{align*}
	d_n\left( x \right) \geqslant \mathrm{e}^{\psi \left( n \right) \left( 1-\epsilon \right)}
\end{align*}
holds for infinitely many $ n $. Since $ \underset{n\rightarrow \infty}{\lim  \mathrm{inf}}\frac{\psi \left( n \right)}{n}=\alpha  $, then there exsits a positive integer $ \widetilde{N}$ such that for any $ n\geqslant \widetilde{N} $, one has
\begin{align*}
	\psi \left( n \right) >n\left( \alpha -\epsilon \right) ,
\end{align*}
thus,
\begin{align*}
	d_n\left( x \right) \geqslant \mathrm{e}^{\psi \left( n \right) \left( 1-\epsilon \right)} \geqslant \mathrm{e}^{n\left( \alpha -\epsilon \right) \left( 1-\epsilon \right)}
\end{align*}
holds for for infinitely many $ n\in \mathbb{N} $.
This show that
\begin{align*}
	E_{\mathrm{sup}}\left( \psi \right)  \subseteq \Big\{ x\in \left( 0,1 \right] \,\,:d_n\left( x \right) \geqslant \mathrm{e}^{n\left( \alpha -\epsilon \right) \left( 1-\epsilon \right)}  \,\,{\rm{for}\,\, \rm{infinitely}\,\,\rm{many}}\,\,n\in \mathbb{N} \Big\} .
\end{align*}
Recall that
\begin{align*}
		H\left( \alpha \right) =\mathrm{dim_H}\,\Big\{ x\in \left( 0,1 \right] \,\,: d_n\left( x \right) \geqslant \mathrm{e}^{\alpha n}\,\,{\rm{for}\,\, \rm{infinitely}\,\,\rm{many}}\,\,n\in \mathbb{N} \Big\} ,
\end{align*}
then
\begin{align*}
\mathrm{dim_H}\,E_{\mathrm{sup}}\left( \psi \right)  \leqslant H\left( \left( \alpha -\epsilon \right) \left( 1-\epsilon \right) \right) .	
\end{align*}
Let $ \epsilon \rightarrow 0^+ $ and by Lemma \ref{lem:H(alpha)}, we have
\begin{align*}
	\mathrm{dim_H}\,E_{\mathrm{sup}}\left( \psi \right)  \leqslant H\left( \alpha \right) .
\end{align*}
\par \textbf{Lower bound} 	Since $ \liminf\limits_{n\rightarrow \infty}\frac{\psi \left( n \right)}{n}=\alpha  $ and $ 0<\alpha <\infty  $, it is easy to see that
\begin{align*}
\liminf \limits_{n\rightarrow \infty}\frac{\psi \left( n \right)}{n\alpha}=\frac{1}{\alpha}\liminf\limits_{n\rightarrow \infty}\frac{\psi \left( n \right)}{n}=1.
\end{align*}

Let  $ \left\{ n_k \right\} _{k\geqslant 1} $ be a subsequence of $ \mathbb{N}  $ such that
\begin{align*}
\underset{k\rightarrow \infty}{\lim}\frac{\psi \left( n_k \right)}{n_k\alpha}=\liminf \limits_{n\rightarrow \infty}\frac{\psi \left( n \right)}{n\alpha}=1.
\end{align*}
Choose a sufficiently sparse subsequence $ \left\{ n_{k_i} \right\} _{i\geqslant 1} $ of  $ \left\{ n_k \right\} _{k\geqslant 1} $ (for simplicity, we still denote by  $ \left\{ n_k \right\} _{k\geqslant 1} $) such that
\begin{align*}
n_1=1\,\,\text{and}\,\,n_1+n_2+\cdots +n_k\leqslant \frac{1}{k+1}n_{k+1}\,\,\left( k\geqslant 2 \right).
\end{align*}
Fix an integer $ M\geqslant 2 $ and for any $ B>1 $, define
\begin{align*}
	F_M\left( B \right) ={} &\Big\{x\in \left( 0,1 \right] :\lfloor B^{n_{k}} \rfloor +2\leqslant d_{n_{k}}\left( x \right) \leqslant 2\lfloor B^{n_{k}} \rfloor +1\,\,{\rm{for\ any}}\,\,k\in \mathbb{N}  \,\,\textup{and} \\
	& \qquad \qquad \qquad 2\leqslant d_j\left( x \right) \leqslant M,\,\,{\rm{if}}\,\, j \ne n_{k}\,\,{\rm{for\ any}}\,\,k\in \mathbb{N}  \Big\}
\end{align*}
and
\begin{align*}
	F_{\alpha}\left( \psi \right) =\Big\{ x\in \left( 0,1 \right] :\underset{n\rightarrow \infty}{\lim  \mathrm{sup}}\frac{\log d_n\left( x \right)}{\psi \left( n \right)}\leqslant 1\,\,\,\text{and}\,\,\,\limsup_{k\rightarrow \infty}\frac{\log d_{n_{k}}\left( x \right)}{\psi \left( n_{k} \right)}\,\,=1\,\, \Big\} .
\end{align*}
For any $ x\in F_M\left(\mathrm{e}^{\alpha} \right)  $, it is obvious that
\begin{align*}
	\underset{n\rightarrow \infty}{\lim  \mathrm{sup}}\frac{\log d_n\left( x \right)}{n\alpha}=1\,\,\,\text{and}\,\, \underset{k\rightarrow \infty}{\lim}\frac{\log d_{n_{k}}\left( x \right)}{n_{k}\alpha}\,\,=1 .
\end{align*}
Thus,
\begin{align*}
	\limsup_{n\rightarrow \infty}\frac{\log d_n\left( x \right)}{\psi \left( n \right)}
	&=\limsup_{n\rightarrow \infty}\frac{\log d_n\left( x \right)}{\psi \left( n \right)}\cdot \liminf_{n\rightarrow \infty}\frac{\psi \left( n \right)}{n\alpha} \\
	& \leqslant \limsup_{n\rightarrow \infty}\frac{\log d_n\left( x \right)}{n\alpha}\\
	&=1
\end{align*}
and
\begin{align*}
	\limsup_{k\rightarrow \infty}\frac{\log d_{n_{k}}\left( x \right)}{\psi \left( n_{k} \right)}\,\,
	&=\limsup_{k\rightarrow \infty}\frac{\log d_{n_{k}}\left( x \right)}{\psi \left( n_{k} \right)}\,\,. \underset{k\rightarrow \infty}{\lim}\frac{\psi \left( n_{k} \right)}{n_{k}\alpha} \\
	&=\limsup_{k\rightarrow \infty}\frac{\log d_{n_{k}}\left( x \right)}{n_{k}\alpha}\\
	&=1 .
\end{align*}
Hence
\begin{align} \label{eq:FM(alpha)}
	F_M\left( \mathrm{e}^{\alpha} \right) \subseteq F_{\alpha}\left( \psi \right) .
\end{align}

Let $ G_M\left( B \right)  $ be the unique solution of the equation
\begin{align*}
	\sum_{k=2}^M{\left( \frac{1}{Bk\left( k-1 \right)} \right) ^s=1},	
\end{align*}
then according to the proof of Theorem 3.2 of \cite{shenHausdorffDimensionSet2017}, for any $ B>1 $, it can be obtained that
\begin{align}\label{eq:H(M,B)}
	\mathrm{dim_H}\,F_M\left( B \right) \geqslant  G_M\left( B \right) \,\,\text{and}\,\,\lim_{M\rightarrow \infty}  G_M\left( B \right) =G\left( B \right) ,
\end{align}
where $ G\left( B \right)  $ is defined in Lemma \ref{lem:H(alpha)}. It is easy to see that $ F_{\alpha}\left( \psi \right) \subseteq E_{\mathrm{sup}}\left( \psi \right)  $,
then by \eqref{eq:FM(alpha)}, we have
\begin{align*}
	\mathrm{dim_H}\,E_{\mathrm{sup}}\left( \psi \right)  \geqslant \mathrm{dim_H}\,F_{\alpha}\left( \psi \right) \geqslant \mathrm{dim_H}\,F_M\left( \mathrm{e}^{\alpha}\right) \geqslant G_M\left( \mathrm{e}^{\alpha} \right).	
\end{align*}

Let $ M\rightarrow \infty  $, it follows from \eqref{eq:H(M,B)} and Lemma \ref{lem:H(alpha)} that
\begin{align*}
	\mathrm{dim_H}\,E_{\mathrm{sup}}\left( \psi \right)  \geqslant \lim_{M\rightarrow \infty} G_M\left( \mathrm{e}^{\alpha} \right)=G\left({\mathrm{e}^{\alpha}}\right) =H\left( \alpha \right) .
\end{align*}
{\bf Case \Rmnum{3}: }$ \underset{n\rightarrow \infty}{\lim}\,\frac{\psi \left( n \right)}{n}=\infty .  $ \par
\textbf{Upper bound }
For any $ 0<\epsilon <1 $, define
\begin{align*}
	\hat{E}\left( \psi ,\epsilon \right)  =\Big\{ x\in \left( 0,1 \right] \,\,: d_n\left( x \right) \geqslant\mathrm{e}^{\psi \left( n \right) \left( 1-\epsilon \right)}\,\,{\rm{for\ infinitely\ many}}\,\,n\in \mathbb{N} \Big\} ,
\end{align*}
it is easy to see that
\begin{align*}
E_{\mathrm{sup}}\left( \psi \right)  \subseteq \hat{E}\left( \psi ,\epsilon \right).  	
\end{align*}
Note that $ \lim\limits_{n\rightarrow \infty} \frac{\psi \left( n \right)}{n}=\infty   $, we have
\begin{align*}
	\lim_{n\rightarrow \infty} \frac{\log  \mathrm{e}^{\psi \left( n \right) \left( 1-\epsilon \right)}}{n} =\lim_{n\rightarrow \infty} \frac{ \psi \left( n \right) \left( 1-\epsilon \right) }{n}=\infty
\end{align*}
and
\begin{align*}
    \liminf \limits_{n\rightarrow \infty}\frac{\log  \log  \mathrm{e}^{\psi \left( n \right) \left( 1-\epsilon \right)}}{n}
    =\liminf \limits_{n\rightarrow \infty}\frac{\log \psi \left( n \right) \left( 1-\epsilon \right) }{n}=\liminf \limits_{n\rightarrow \infty}\frac{\log \psi \left( n \right)}{n}.	
\end{align*}
Recall that
\begin{align*}
		\log A=\liminf \limits_{n\rightarrow \infty}\frac{\log \psi \left( n \right)}{n},
\end{align*}
then it follows from  Theorem 4.2 of \cite{shenHausdorffDimensionSet2017} that
\begin{align*}
	\mathrm{dim_H}\,E_{\mathrm{sup}}\left( \psi \right)  \leqslant \mathrm{dim_H}\,\hat{E}\left( \psi ,\epsilon \right)  =\frac{1}{1+A}.
\end{align*}
\par

\textbf{Lower bound}
 We obtain the lower bound of the Hausdorff dimension of $ E_{\mathrm{sup}}\left( \psi \right) $ by constructing a suitable Cantor subset of $ E_{\mathrm{sup}}\left( \psi \right) $.

 Define $ \theta :\mathbb{N} \rightarrow \mathbb{R} ^+ $  as follows:
 \begin{align*}
	\theta \left( n \right) =\underset{k\geqslant n}{\min}\psi \left( k \right) .
\end{align*}
Since $ \lim\limits_{n\rightarrow \infty} \psi \left( n \right)=\infty  $, $ \theta \left( n \right)  $ is well defined .
Secondly, for any $ \epsilon >0 $, define a sequence $ \left\{ r_n \right\} _{n\geqslant 1} $ as follows:
\begin{align}\label{eq:r(n)}
	r_1=e^{\theta \left( 1 \right)} \,\,\text{and}\,\,r_n=\min \left\{ \mathrm{e}^{\theta \left( n \right)} ,\prod_{i=1}^{n-1}{r_{i}^{A-1+\epsilon}} \right\} \,\,\left(  n\geqslant 2 \right) .
\end{align}
It follows from the proof of Theorem 1.1 of \cite{lulufangExceptionalSetsBorel2021} that \begin{align}\label{eq:limsup r(n)}
	\mathop {\lim\mathrm{sup}} \limits_{n\rightarrow \infty}\frac{\log r_n}{\psi \left( n \right)}=1\,\,,\,\,\lim_{n\rightarrow \infty} r_n=\infty \,\,\text{and}\,\,\mathop {\lim\mathrm{sup}} \limits_{n\rightarrow \infty}\frac{\log r_{n+1}}{\log (r_1\cdots r_n)}\leqslant A-1+\epsilon .
\end{align}

We use the sequence $ \left\{ r_n \right\} _{n\geqslant 1} $ to construct the Cantor subset of $ E_{\mathrm{sup}}\left( \psi \right)  $. By \eqref{eq:limsup r(n)}, there exsits a positive interger $ M $ such that $ M\lfloor r_n+1 \rfloor \geqslant 4 $ for all $ n\in \mathbb{N}  $. Define
\begin{align*}
	E\big( \left\{ r_n \right\} _{n\geqslant 1} \big)  =\Big\{ x\in \left( 0,1 \right] :M\lfloor r_n+1 \rfloor \leqslant d_n\left( x \right) \leqslant 2M\lfloor r_n+1 \rfloor -1 \, \text{for}\,\,\text{any} \,\,n\in \mathbb{N} \Big\} .
\end{align*}
For any $ x\in E\big( \left\{ r_n \right\} _{n\geqslant 1} \big)    $ , by  \eqref{eq:limsup r(n)} , one has
\begin{align}\label{eq:baohan}
	1=\mathop {\lim  \mathrm{sup}} \limits_{n\rightarrow \infty}\frac{\log Mr_n}{\psi \left( n \right)}\leqslant \mathop {\lim  \mathrm{sup}} \limits_{n\rightarrow \infty}\frac{\log d_n\left( x \right)}{\psi \left( n \right)}\leqslant \mathop {\lim  \mathrm{sup}} \limits_{n\rightarrow \infty}\frac{\log 2M\left( r_n+1 \right)}{\psi \left( n \right)}=1.
\end{align}
This shows that
\begin{align*}
	E\big( \left\{ r_n \right\} _{n\geqslant 1} \big)  \subseteq E_{\mathrm{sup}}\left( \psi \right)  .
\end{align*}
Note that
\begin{align*}
	\limsup_{n\rightarrow \infty}\frac{\log M\lfloor r_{n+1}+1 \rfloor}{\log M^n \cdot \prod\limits_{i=1}^n{\lfloor r_i+1 \rfloor}}
	& \leqslant \limsup \limits_{n\rightarrow \infty}\frac{\log M\left( r_{n+1}+1 \right)}{\log M^n\cdot\prod\limits_{i=1}^n{r_i}}\\
	&=\limsup \limits_{n\rightarrow \infty}\frac{\log M+\log \left( r_{n+1}+1 \right)}{n\log M+\log r_1\cdots r_n}\\
	&=\limsup \limits_{n\rightarrow \infty}\frac{\log \left( r_{n+1}+1 \right)}{n\log M+\log r_1\cdots r_n}\\
	&=\limsup \limits_{n\rightarrow \infty}\frac{\log \left( r_{n+1}+1 \right)}{\log r_1\cdots r_n}\cdot \frac{1}{1+\frac{n\log M}{\log r_1\cdots r_n}}\\
	&=\limsup \limits_{n\rightarrow \infty}\frac{\log \left( r_{n+1}+1 \right)}{\log r_1\cdots r_n} \\
	&\leqslant A-1+\epsilon .
\end{align*}
Then by Lemma \ref{lem:key lem}, we have
\begin{align}
	\label{eq:Es{r(n)}}
	\mathrm{dim_H}\,E\big( \left\{ r_n \right\} _{n\geqslant 1} \big)
	&=\liminf \limits_{n\rightarrow \infty}\frac{\log M^n\cdot\prod\limits_{i=1}^n{\lfloor r_i+1 \rfloor}}{2\log M^n\cdot\prod\limits_{i=1}^n{\lfloor r_i+1 \rfloor}+\log M\lfloor r_{n+1}+1 \rfloor} \notag \\
	&=\frac{1}{2+\limsup \limits_{n\rightarrow \infty}\dfrac{\log M\lfloor r_{n+1}+1 \rfloor}{\log M^n\cdot\prod\limits_{i=1}^n{\lfloor r_i+1 \rfloor}}} \geqslant \frac{1}{1+A+\epsilon}.
\end{align}
Thus,
\begin{align*}
	\mathrm{dim_H}\,E_{\mathrm{sup}}\left( \psi \right)  \geqslant \mathrm{dim_H}\,E\big( \left\{ r_n \right\} _{n\geqslant 1} \big)  \geqslant \frac{1}{1+A+\epsilon}.
\end{align*}
Let $ \epsilon \rightarrow 0^+ $, we have
\begin{align*}
	\mathrm{dim_H}\,E_{\mathrm{sup}}\left( \psi \right)  \geqslant \frac{1}{1+A}.
\end{align*}

\section{PROOF OF THEOREM \ref{Thm 1.2} AND THEOREM \ref{Thm 1.3}}
In this section, we complete the proofs of Theorem \ref{Thm 1.2} and Theorem \ref{Thm 1.3}.
\subsection{Proof of Theorem \ref{Thm 1.2}}
We split the proof into two parts.\par
\textbf{Lower bound} We construct a suitable Cantor subset of $ E\left( \psi \right)  $ to get the lower bound of the Hausdorff dimension of $ E\left( \psi \right)  $. Since $ \lim\limits_{n\rightarrow \infty} \psi \left( n \right) =\infty  $ and $ \psi \left( n \right) >0 $ for any $ n\in\mathbb{N}  $, there exsits an interger $ M $ such that $ M\lfloor \mathrm{e}^{\psi \left( n \right)}+1 \rfloor \geqslant 4  $ holds for any $ n\in \mathbb{N}  $. Define
\begin{align*}
	W \left( \psi \right) =\left\{ x\in \left( 0,1 \right] :M\lfloor \mathrm{e}^{\psi \left( n \right)}+1 \rfloor \leqslant d_n\left( x \right) \leqslant 2M\lfloor \mathrm{e}^{\psi \left( n \right)}+1 \rfloor -1  \, \,{\rm{for}\,\, \rm{any}}\,\, n\in \mathbb{N} \right\} .
\end{align*}
Similar to \eqref{eq:baohan}, it is easy to see that
\begin{align*}
	W \left( \psi \right) \subseteq E\left( \psi \right) .
\end{align*}
Therefore by Lemma \ref{lem:key lem}, similar to \eqref{eq:Es{r(n)}}, we have
\begin{align*}
	\mathrm{dim_H}\,W \left( \psi \right) &\geqslant \dfrac{1}{2+\mathop {\lim\mathrm{sup}} \limits_{n\rightarrow \infty}\frac{\log ( e^{\psi \left( n+1 \right)}+1 )}{\log  e^{\psi \left( 1 \right) +\cdots +\psi \left( n \right)}}}
	=\frac{1}{2+\mathop {\lim  \mathrm{sup}} \limits_{n\rightarrow \infty}\frac{\psi \left( n+1 \right)}{\psi \left( 1 \right) +\cdots +\psi \left( n \right)}}.
\end{align*}
Recall that
\begin{align*}
	\eta =\mathop {\lim  \mathrm{sup}} \limits_{n\rightarrow \infty}\frac{\psi \left( n+1 \right)}{\psi \left( 1 \right) +\cdots +\psi \left( n \right)},
\end{align*}
hence
\begin{align*}
	\mathrm{dim_H}\,E\left( \psi \right) \geqslant \mathrm{dim_H}\,W \left( \psi \right) \geqslant \frac{1}{2+\eta }.
\end{align*}
\par
\textbf{Upper bound}  For any $ 0<\epsilon <1 $ and $ r\in \mathbb{N}  $, we define
\begin{align*}
	E_{\psi}\left( r,\epsilon \right)  =\Big\{ x\in \left( 0,1 \right] :\mathrm{e}^{\psi \left( n \right) \left( 1-\epsilon \right)} +1\leqslant d_n\left( x \right) \leqslant \mathrm{e}^{\psi \left( n \right) \left( 1+\epsilon \right)} +1\, \,{\rm{for}\,\, \rm{any}}\,\,  n\geqslant r \Big\} .
\end{align*}
It is easy to see that
\begin{align*}
	E\left( \psi \right) \subseteq \bigcup_{r=1}^{\infty}{E_{\psi}\left( r,\epsilon \right) } 	,
\end{align*}
then we have
\begin{align}
\label{eq:Hd of E(phi) fangsuo}
	\mathrm{dim_H}\,E\left( \psi \right) \leqslant \mathrm{dim_H}\,\bigcup_{r=1}^{\infty}{E_{\psi}\left( r,\epsilon \right)}=\mathop {\mathrm{sup}} \limits_{r\geqslant 1}\mathrm{dim_H}\,E_{\psi}\left( r,\epsilon \right)  .	
\end{align}

Now we consider the Hausdorff dimension of $ E_{\psi}\left( r,\epsilon \right)   $ for any $ r\in \mathbb{N}  $. Let $ M\geqslant 1 $ be an integer, and define $ 	D_n\left( \epsilon ,M,r \right)  $ as the collection of sequence $ \left( d_1,\cdots ,d_n \right)  $ satisfying
\begin{align*}
	\begin{cases}
		 2\leqslant d_k\leqslant M+1 & \text{if }  1\leqslant  k < r,\\
	\mathrm{e}^{\psi \left( k \right) \left( 1-\epsilon \right)}+1\leqslant d_k\leqslant \mathrm{e}^{\psi \left( k \right) \left( 1+\epsilon \right)}+1& \text{if } r \leqslant  k \leqslant  n.
	\end{cases}
\end{align*}
By the construction of $ 	D_n\left( \epsilon ,M,r \right)  $, it is clear that
\begin{align}\label{eq:number of D(e,m)}
	\#D_n\left( \epsilon ,M,r \right)  &\leqslant M^{r-1}\prod_{k=r}^n{2\epsilon \psi \left( k \right)}\cdot \mathrm{e}^{\left( 1+\epsilon \right) \psi \left( k \right)}.
\end{align}
For any $\left( \sigma _1,\cdots ,\sigma _n \right) \in D_n\left( \epsilon ,M,r \right)   $, put
\begin{align*}
	J_n(\sigma _1,\cdots ,\sigma _n)=\bigcup_{ (\sigma _1,\cdots ,\sigma _n,\sigma _{n+1})\in D_{n+1}\left( \epsilon ,M,r \right) }{I_{n+1}(\sigma _1,\cdots ,\sigma _n,\sigma _{n+1})}.
\end{align*}
We define
\begin{align*}
	\tilde{E}_{\psi ,r,\epsilon}\left( M \right)    ={}
	&\Big\{ x\in \left( 0,1 \right] :\mathrm{e}^{\psi \left( k \right) \left( 1-\epsilon \right)} +1\leqslant d_k\left( x \right) \leqslant \mathrm{e}^{\psi \left( k \right) \left( 1+\epsilon \right)} +1 \,\,\text{for}\,\,{\rm{any}}\ k\geqslant r \\ &\qquad\qquad\quad\text{and}\,\,2\leqslant d_k\left( x \right) \leqslant M+1\,\,\text{for}\  1\leqslant k<r \Big\} .
\end{align*}
Clearly,
\begin{align*}
	E_{\psi}\left( r,\epsilon \right) \subseteq \bigcup_{M=1}^{\infty}{\tilde{E}_{\psi ,r,\epsilon}\left( M \right)},
\end{align*}
then
\begin{align}\label{eq: dim E(phi) of sup}
	\mathrm{dim_H}\,E_{\psi}\left( r,\epsilon \right) \leqslant \mathrm{dim_H}\,\bigcup_{M=1}^{\infty}{\tilde{E}_{\psi ,r,\epsilon}\left( M \right)}=\underset{M\geqslant 1}{\mathrm{sup}}\mathrm{dim_H}\,\tilde{E}_{\psi ,r,\epsilon}\left( M \right) .
\end{align}
For any $ M\geqslant 1 $, now we estimate the Hausdorff dimension of $ \tilde{E}_{\psi ,r,\epsilon}\left( M \right)   $. It is clear that  for any $ n\geqslant r $, one has
\begin{align*}
	\tilde{E}_{\psi ,r,\epsilon}\left( M \right)  \subseteq  \bigcup_{\left( \sigma _1,\cdots ,\sigma _n \right) \in D_n\left( \epsilon ,M,r \right)}{J_n(\sigma _1,\cdots ,\sigma _n)}.
\end{align*}
By Propositon \ref{proposition: lenth of In}, we have
\begin{align*}
	\big| I_{n+1}(\sigma _1,\cdots ,\sigma _n,\sigma _{n+1}) \big|=\frac{1}{Q_{n+1}(\sigma _1,\cdots ,\sigma _n,\sigma _{n+1})}\leqslant \frac{1}{\prod\limits_{i=1}^{n+1}{\left( \sigma _i-1 \right) ^2}},
\end{align*}
then
\begin{align}\label{eq:Jn}
	\big| J_n(\sigma _1,\cdots ,\sigma _n) \big|
	& \leqslant \sum_{\sigma _{n+1} \geqslant \mathrm{e}^{\psi \left( n+1 \right) \left( 1-\epsilon \right)}+1}{\big| I_{n+1}(\sigma _1,\cdots ,\sigma _n,\sigma _{n+1}) \big|} \notag \\
	& \leqslant \sum_{\sigma _{n+1}\geqslant \mathrm{e}^{\psi \left( n+1 \right) \left( 1-\epsilon \right)}+1}{\prod\limits_{i=1}^{n+1}{\frac{1}{\left( \sigma _i-1 \right) ^2}}} \notag  \\
	& =\frac{1}{\prod\limits_{i=1}^n{\left( \sigma _i-1 \right) ^2}} \sum_{\sigma _{n+1}\geqslant \mathrm{e}^{\psi \left( n+1 \right) \left( 1-\epsilon \right)}+1}{\frac{1}{\left( \sigma _{n+1}-1 \right) ^2}}.
\end{align}
Since $ \lim\limits_{n\rightarrow \infty} \psi \left( n \right) =\infty  $, there exists an  integer $ \hat{N}\in \mathbb{N}  $ such that for any $ n>\hat{N} $, it holds that  $ \mathrm{e}^{\psi \left( n \right) \left( 1-\epsilon \right)} >3 $ .
Thus, for $ n>\hat{N} $, one has
\begin{align} \label{eq:sigma n+1}
	\sum_{\sigma _{n+1}\geqslant \mathrm{e}^{\psi \left( n+1 \right) \left( 1-\epsilon \right)}+1}{\frac{1}{\left( \sigma _{n+1}-1 \right) ^2}}
	&\leqslant \sum_{\sigma _{n+1}-1\geqslant \mathrm{e}^{\psi \left( n+1 \right) \left( 1-\epsilon \right)}}{\frac{1}{\left( \sigma _{n+1}-1 \right) \left( \sigma _{n+1}-2 \right)}} \notag \\
	&=\sum_{\sigma _{n+1}-1\geqslant \mathrm{e}^{\psi \left( n+1 \right) \left( 1-\epsilon \right)}}{\frac{1}{\left( \sigma _{n+1}-2 \right)}-\frac{1}{\left( \sigma _{n+1}-1 \right)}} \notag \\
	&\leqslant \frac{1}{\mathrm{e}^{\psi \left( n+1 \right) \left( 1-\epsilon \right)} -1}.
\end{align}
For $ \left( \sigma _1,\cdots ,\sigma _n \right) \in D_n\left( \epsilon ,M,r \right)  $, one has
\begin{align*}
	\begin{cases}
		2\leqslant \sigma _i\leqslant M & \text{if } 1\leqslant i\leqslant r,\\
		\sigma _i-1\geqslant \mathrm{e}^{\psi \left( i \right) \left( 1-\epsilon \right)} & \text{if } r\leqslant i\leqslant n.
    \end{cases}
\end{align*}
Then, by the formula \eqref{eq:Jn} and \eqref{eq:sigma n+1}, when $ n>\hat{N} $, we have
\begin{align}\label{JMAX}
	\big| J_n(\sigma _1,\cdots ,\sigma _n) \big| &\leqslant \frac{1}{\prod\limits_{i=r}^n{\mathrm{e}^{2\psi \left( i \right) \left( 1-\epsilon \right) }}}\cdot \sum_{\sigma _{n+1} \geqslant \mathrm{e}^{\psi \left( n+1 \right) \left( 1-\epsilon \right)}+1}{\frac{1}{\left( \sigma _{n+1}-1 \right) ^2}} \notag\\
	&\leqslant \frac{1}{\mathrm{e}^{\psi \left( n+1 \right) \left( 1-\epsilon \right)}-1}\cdot \prod_{i=r}^n{\mathrm{e}^{-2\psi \left( i \right) \left( 1-\epsilon \right) }}.
\end{align}
For convenience, we write
\begin{align*}
	\big| J_n(\tilde{\sigma}_1,\cdots ,\tilde{\sigma}_n) \big|
	:=\underset{\left( \sigma _1,\cdots ,\sigma _n \right) \in D_n\left( \epsilon ,M,r \right)}{\max}\big| J_n(\sigma _1,\cdots ,\sigma _n) \big| .	
\end{align*}
It follows from $(\ref{JMAX})$ and Proposition 4.1 of \cite{falconer2014fractal} that
\begin{align*}
	\underline{\mathrm{dim}}_B\tilde{E}_{\psi ,r,\epsilon}\left( M \right)
	& \leqslant \mathop {\lim\mathrm{inf}} \limits_{n\rightarrow \infty}\frac{\log \#D_n\left( \epsilon ,M,r \right)}{-\log \big| J_n(\tilde{\sigma}_1,\cdots ,\tilde{\sigma}_n) \big|}	\\
	& \leqslant \liminf \limits_{n\rightarrow \infty}\frac{\left( 1+\epsilon \right) \sum\limits_{k=r}^n{\psi \left( k \right)}+\left( r-1 \right) \log M+\sum\limits_{k=r}^n{\log \left( 2\epsilon \psi \left( k \right) \right)}}{2\left( 1-\epsilon \right) \sum\limits_{k=r}^n{\psi \left( k \right)}+\log \left( \mathrm{e}^{\psi \left( n+1 \right) \left( 1-\epsilon \right)}-1\right)} \\
	&=\liminf \limits_{n\rightarrow \infty}\frac{1+\epsilon}{1-\epsilon}\cdot \frac{\sum\limits_{k=r}^n{\psi \left( k \right)}}{2\sum\limits_{k=r}^n{\psi \left( k \right)}+\psi \left( n+1 \right)} \\
	&=\frac{1+\epsilon}{1-\epsilon}\cdot \mathop {\lim\mathrm{inf}} \limits_{n\rightarrow \infty}\frac{1}{2+\frac{\psi \left( n+1 \right)}{\psi \left( r \right) +\cdots +\psi \left( n \right)}} \\
	& =\frac{1+\epsilon}{1-\epsilon}\cdot \frac{1}{2+\mathop {\lim  \mathrm{sup}} \limits_{n\rightarrow \infty}\frac{\psi \left( n+1 \right)}{\psi \left( r\right) +\cdots +\psi \left( n \right)}}\\
	&\leqslant \frac{1+\epsilon}{1-\epsilon}\cdot \frac{1}{2+\mathop {\lim\mathrm{sup}} \limits_{n\rightarrow \infty}\frac{\psi \left( n+1 \right)}{\psi \left( 1 \right) +\cdots +\psi \left( n \right)}}\\
	&= \frac{1+\epsilon}{1-\epsilon}\cdot \frac{1}{2+\eta} ,
\end{align*}
then
\begin{align*}
	\mathrm{dim_H}\,\tilde{E}_{\psi ,r,\epsilon}\left( M \right) \leqslant \underline{\mathrm{dim}}_B\tilde{E}_{\psi ,r,\epsilon}\left( M \right) \leqslant  \frac{1+\epsilon}{1-\epsilon}\cdot \frac{1}{2+\eta}.	
\end{align*}
By \eqref{eq:Hd of E(phi) fangsuo} and \eqref{eq: dim E(phi) of sup}, we have
\begin{align*}
	\mathrm{dim_H}\,E\left( \psi \right)  \leqslant \frac{1+\epsilon}{1-\epsilon}\cdot \frac{1}{2+\eta}.	
\end{align*}
By letting $ \epsilon \rightarrow 0^+ $, we can obtain that
\begin{align*}
	\mathrm{dim_H}\,E\left( \psi \right)  \leqslant \frac{1}{2+\eta }.
\end{align*}

\subsection{Proof of Theorem \ref{Thm 1.3}}
Like before, we split the proof into two parts.\par
\textbf{Lower bound} Recall that
\begin{align*}
	\log V =\mathop {\lim  \mathrm{sup}} \limits_{n\rightarrow \infty}\frac{\log \psi \left( n \right)}{n}.
\end{align*}
If $ V=\infty  $, then $ \mathrm{dim_H}\,E_{\mathrm{inf}}\left( \psi \right) \geqslant \frac{1}{1+V} $ holds trivially. Now we assume $ 1\leqslant V<\infty  $, the lower bound of the Hausdorff dimension of $ E_{\mathrm{inf}}\left( \psi \right)   $ is yielded by constructing a suitable Cantor subset of $ E_{\mathrm{inf}}\left( \psi \right)   $  .

 For any $ 0<\epsilon <1 $, define a sequence $\left\{ L_n\left( \epsilon \right) \right\}_{n\ge1}  $ as follows:
\begin{align*}
	L_n\left( \epsilon \right) =\underset{j\geqslant n}{\mathrm{sup}}\left\{ \mathrm{e}^{\psi \left( j \right) \left( V+\epsilon \right) ^{n-j}}\right\} \,\,{\rm{for}\,\,\rm{any}} \,\, n\in \mathbb{N} .	
\end{align*}
It follows from the proof of Theorem 1.3 of \cite{lulufangExceptionalSetsBorel2021} that
\begin{align} \label{eq:L_j}
	\liminf \limits_{n\rightarrow \infty}\frac{\log L_n\left( \epsilon \right)}{\psi \left( n \right)}=1 \,\,\text{and}\,\, \log L_{n+1}\left( \epsilon \right) -\log L_1\left( \epsilon \right) \leqslant \left( V+\epsilon -1 \right) \sum_{i=1}^n\log{L_i\left( \epsilon \right)}.
\end{align}

Now, we use the sequence  $\left\{ L_n\left( \epsilon \right) \right\}_{n\ge1}  $ to construct the Cantor subset of $ E_{\mathrm{inf}}\left( \psi \right)  $. Note that for any $ n\in \mathbb{N}  $, $L_n\left( \epsilon \right) \geqslant \mathrm{e}^{\psi \left( n \right)}  $, which implies $ \lim\limits_{n\rightarrow \infty} L_n\left( \epsilon \right) =\infty  $, then there exsits a positive interger $ M $ such that $ M\lfloor L_n\left( \epsilon \right) +1 \rfloor \geqslant 4 $ holds for all $ n\in \mathbb{N}  $. Define
\begin{align*}
	E\left( \left\{ L_n\left( \epsilon \right) \right\} _{n\geqslant 1} \right) =\Big\{ x\in \left( 0,1 \right] :M\lfloor L_n\left( \epsilon \right) +1 \rfloor \leqslant d_n\left( x \right) \leqslant 2M\lfloor L_n\left( \epsilon \right) +1 \rfloor -1  \,\,{\rm{for}\,\,\rm{any}} \,\,n\in \mathbb{N} \Big\} .
\end{align*}
Similar to \eqref{eq:baohan}, it is easy to verify that
\begin{align*}
	E\left( \left\{ L_n\left( \epsilon \right) \right\} _{n\geqslant 1} \right) \subseteq E_{\mathrm{inf}}\left( \psi \right)  .
\end{align*}
Therefore by Lemma \ref{lem:key lem} and formula \eqref{eq:L_j}, similar to \eqref{eq:Es{r(n)}}, we have
\begin{align*}
	\mathrm{dim_H}\,E\left( \left\{ L_n\left( \epsilon \right) \right\} _{n\geqslant 1} \right) \geqslant \frac{1}{2+\mathop {\lim\mathrm{sup}} \limits_{n\rightarrow \infty}\frac{\log \left( L_{n+1}\left( \epsilon \right) +1 \right)}{\log L_1\left( \epsilon \right) +\cdots +L_n\left( \epsilon \right)}}\geqslant \frac{1}{1+V+\epsilon}.	
\end{align*}
Hence
\begin{align*}
	\mathrm{dim_H}\,E_{\mathrm{inf}}\left( \psi \right)  \geqslant \frac{1}{1+V+\epsilon}.	
\end{align*}
Let $ \epsilon \rightarrow 0^+ $, this shows that
\begin{align*}
	\mathrm{dim_H}\,E_{\mathrm{inf}}\left( \psi \right) \geqslant \frac{1}{1+V}.
\end{align*}
\par
\textbf{Upper bound}  For any $ 0<\epsilon <1 $ and $ r\in \mathbb{N}  $, we define
\begin{align*}
	F_{\psi}\left( r,\epsilon \right) =\Big\{ x\in \left( 0,1 \right] :d_n\left( x \right) \geqslant \mathrm{e}^{\psi \left( n \right) \left( 1-\epsilon \right)}\,\,{\rm{for}\,\,\rm{any}} \,\,n \geqslant r \Big\} .
\end{align*}
It is easy to see that
\begin{align*}
	E_{\mathrm{inf}}\left( \psi \right)  \subseteq \bigcup_{r=1}^{\infty}{F_{\psi}\left( r,\epsilon \right)}	,
\end{align*}
then
\begin{align}\label{eq:Hd of F(phi) fangsuo}
	\mathrm{dim_H}\,E_{\mathrm{inf}}\left( \psi \right) \leqslant \mathrm{dim_H}\,\bigcup_{r=1}^{\infty}{F_{\psi}\left( r,\epsilon \right)}=\mathop {\mathrm{sup}} \limits_{r\geqslant 1}\mathrm{dim_H}\,F_{\psi}\left( r,\epsilon \right) .
\end{align}

For any $ r\in \mathbb{N}  $, now we estimate the Hausedorff dimension of $ F_{\psi}\left( r,\epsilon \right)   $,
we consider the following cases. \\
$ \left( \mathrm{i} \right)  $  $ V=1 $. Since $ \lim\limits_{n\rightarrow \infty} \psi \left( n \right) =\infty  $,  it is obvious that
\begin{align*}
	F_{\psi}\left( r,\epsilon \right) \subseteq \left\{ x\in \left( 0,1 \right] :\lim_{n\rightarrow \infty} d_n\left( x \right) =\infty \right\} .
\end{align*}
By Theorem 3.2 of \cite{shen2011fractional}, we have
\begin{align*}
	\mathrm{dim_H}\,F_{\psi}\left( r,\epsilon \right) \leqslant \mathrm{dim_H}\,\left\{ x\in \left( 0,1 \right] :\lim_{n\rightarrow \infty} d_n\left( x \right) =\infty \right\} \,\,=\frac{1}{2}=\frac{1}{1+V}.
\end{align*}
$ \left( \mathrm{ii} \right)  $  $ 1<V<\infty  $. By the definition of $ V $, one has
\begin{align*}
	\psi \left( n \right) \geqslant \left( V-\epsilon \right) ^n
\end{align*}
holds for infinitely  many $ n\in \mathbb{N}  $. We denote
\begin{align*}
U \big( \mathrm{e}^{\left( 1-\epsilon \right)},V-\epsilon \big) = \Big\{ x\in \left( 0,1 \right] :d_n\left( x \right) \geqslant \mathrm{e}^{\left( 1-\epsilon \right) \left( V-\epsilon \right) ^n}\,\, {\rm{for}\,\, \rm{infinitely}\,\,\rm{many}}\,\,n\in \mathbb{N}  \Big\} ,
\end{align*}
then
\begin{align*}
	F_{\psi}\left( r,\epsilon \right) \subseteq U \big( \mathrm{e}^{\left( 1-\epsilon \right)},V-\epsilon \big) .
\end{align*}
Supposed that $ 0<\epsilon <\min \left\{ V-1,1 \right\}  $, then by Theorem 3.1 of \cite{shen2011fractional}, we have
\begin{align*}
	\mathrm{dim_H}\,F_{\psi}\left( r,\epsilon \right) \leqslant \mathrm{dim_H}\,U \big( \mathrm{e}^{\left( 1-\epsilon \right)},V-\epsilon \big) =\frac{1}{1+V-\epsilon}.
\end{align*}
$ \left( \mathrm{iii} \right)  $ $ V=\infty  $. By the definition of $ V $, for any $ M>1  $, there exists infinitely  many $ n\in \mathbb{N}  $ such that
\begin{align*}
	\psi \left( n \right) \geqslant M^n ,
\end{align*}
then
\begin{align*}
	F_{\psi}\left( r,\epsilon \right) \subseteq \Big\{ x\in \left( 0,1 \right] :d_n\left( x \right) \geqslant\mathrm{e}^{\left( 1-\epsilon \right) M^n} \,\, {\rm{for}\,\, \rm{infinitely}\,\,\rm{many}}\,\,n\in \mathbb{N} \Big\} .
\end{align*}
Similar to case $ \left( \mathrm{ii} \right)  $, we have
\begin{align*}
	\mathrm{dim_H}\,F_{\psi}\left( r,\epsilon \right) \leqslant \frac{1}{1+M}.
\end{align*}
By letting $ M\rightarrow \infty  $, one has
\begin{align*}
	\mathrm{dim_H}\,F_{\psi}\left( r,\epsilon \right) \leqslant \lim_{M\rightarrow \infty} \frac{1}{1+M}=0 =\frac{1}{1+V}.
\end{align*}
\par
Combining $ \left( \mathrm{i} \right) -\left( \mathrm{iii} \right)  $ and \eqref{eq:Hd of F(phi) fangsuo} , it follows that
\begin{align*}
		\mathrm{dim_H}\,E_{\mathrm{inf}}\left( \psi \right) \leqslant \underset{r\geqslant 1}{\mathrm{sup}}\,\,\mathrm{dim_H}\,F_{\psi}\left( r,\epsilon \right)   \leqslant
	\begin{cases}
		\frac{1}{1+V} & \text{if } V=1,\\
		\frac{1}{1+V-\epsilon} & \text{if } 1<V<\infty\,\, \text{and}\,\,0<\epsilon <\min \left\{ V-1,1 \right\},\\
		\frac{1}{1+V}  & \text{if } V=\infty \,\,  \text{and}\,\,0<\epsilon <1.
	\end{cases}
\end{align*}
By letting $ \epsilon \rightarrow 0^+ $, we conclude that
\begin{align*}
	\mathrm{dim_H}\,E_{\mathrm{inf}}\left( \psi \right) \leqslant \frac{1}{1+V}.
\end{align*}



\end{document}